\documentclass[10pt]{amsart}
\usepackage{amsmath, amscd, amssymb, amsthm, verbatim}

\newtheorem{theorem}{Theorem}[section]
\newtheorem{corollary}[theorem]{Corollary}
\newtheorem{lemma}[theorem]{Lemma}
\newtheorem{proposition}[theorem]{Proposition}
\newtheorem{remark}[theorem]{Remark}

\newtheorem{definition}[theorem]{Definition}

\DeclareMathOperator{\Rc}{Rc}
\DeclareMathOperator{\Div}{Div}

\newcommand{\E}{\mathcal{E}}
\newcommand{\F}{\mathcal{F}}
\newcommand{\W}{\mathcal{W}}
\newcommand{\R}{\mathbb{R}}

\begin{document}
\title[Entropy and eigenvalue on evolving manifolds]{Entropy and
  lowest eigenvalue \\ on evolving manifolds}

\author{Hongxin Guo}
\address{School of mathematics and information science, Wenzhou
  University, Wenzhou, Zhejiang 325035, China} \email{guo@wzu.edu.cn;
  hongxin.guo@uni.lu}

\author{Robert Philipowski \and Anton Thalmaier}
\address{Mathematics Research Unit, FSTC, University of Luxembourg, 6
  rue Richard Cou\-den\-hove-Kalergi, L-1359 Luxembourg, Grand-Duchy
  of Luxembourg} \email{robert.philipowski@uni.lu;
  anton.thalmaier@uni.lu}

\thanks{Acknowledgements. Research supported by NSF of China (grants
  no. 11001203 and 11171143) and Fonds National de la Recherche
  Luxembourg.}
\date{\today}

\begin{abstract}
  In this note we determine the first two derivatives of the classical
  Boltzmann-Shannon entropy of the conjugate heat equation on general
  evolving manifolds. Based on the second derivative of the
  Boltzmann-Shannon entropy, we construct Perelman's $\F$ and $\W$
  entropy in abstract geometric flows. Monotonicity of the entropies
  holds when a technical condition is satisfied.

  This condition is satisfied on static Riemannian manifolds with
  nonnegative Ricci curvature, for Hamilton's Ricci flow, List's
  extended Ricci flow, M\"{u}ller's Ricci flow coupled with harmonic
  map flow and Lorentzian mean curvature flow when the ambient space
  has nonnegative sectional curvature.

  Under the extra assumption that the lowest eigenvalue is
  differentiable along time, we derive an explicit formula for the
  evolution of the lowest eigenvalue of the Laplace-Beltrami operator
  with potential in the abstract setting.

\end{abstract}

\maketitle
\section{introduction}
\subsection{Introduction}
Geometric flows have been studied extensively. The idea is to evolve
metrics in certain ways usually by heat type equations to obtain
better metrics on manifolds and thus to gain topological information
of the manifolds. It is desirable to derive evolution equations in a
general setting such that the formulas may be applied to various
flows. For instance, very nice general approaches to get monotone
quantities on evolving manifolds have been developed in
\cite{E-K-N-T,M}.

We briefly introduce notations of an abstract geometric flow.  Let $M$
be an $n$-dimensional compact manifold. Assume that $\alpha(t,y)$ is a
time-dependent symmetric two-tensor on $M$, and that $g(t)$ is a
family of one parameter Riemannian metrics evolving along the flow
equation
\begin{equation}\label{metric evolving}
  \frac{\partial g}{\partial t}=-2\alpha,\quad t\in (0, T),
\end{equation}
where $T$ is some fixed positive constant. Let $A:= g^{ij}\alpha_{ij}$
be the trace of $\alpha$ with respect to $g(t)$.

Classical quantities on static manifolds have nice applications on
evolving manifolds by certain natural modifications.  Such a quantity
is the Boltzmann-Shannon entropy for heat equation. Formally, the
conjugate of the heat operator $\frac{\partial}{\partial t}-\Delta$ on
space-time is $-\frac{\partial}{\partial t}-\Delta +A$. It turns out
as Perelman \cite{P} shows, on evolving manifolds it is natural to
work with the entropy for the conjugate heat equation. We will derive
the first two derivatives of Boltzmann-Shannon entropy for the
conjugate heat equation, and based on that we define Perelman's $\F$
and $\W$ entropy in the framework of abstract geometric flows.

Other classical quantities on static Riemannian manifolds are the
eigenvalues of the Laplace-Beltrami operator $\Delta$. When the metric
evolves, it is natural to include a potential function. Perelman
\cite{P} shows that the lowest eigenvalue of $-\Delta+ R/4$ is
monotone nondecreasing along the Ricci flow. Furthermore by deriving
explicit formula of the derivative, Cao \cite{Cao1, Cao2} shows that
the monotonicity holds for the lowest eigenvalue of $-\Delta+cR$ for
any $c\ge 1/4$, see also Li \cite{Li}.

In \cite{M} Reto M\"{u}ller derived formulas for the reduced volume in
abstract geometric flows.  His formulation is very general and thus
can be applied to different flows. He shows that the reduced volume is
monotone when a technical assumption holds, which is satisfied for
static manifolds with positive Ricci curvature, Hamilton's Ricci flow,
List's extended Ricci flow, M\"{u}ller's Ricci flow coupled with
harmonic map flow and Lorentzian mean curvature flow when the ambient
manifold has nonnegative sectional curvature.  This allows him to
establish new monotonicity formulas for these flows.

One of our purposes in this paper is to show that the same phenomena
as for reduced volume holds for entropy and eigenvalues.

\subsection{Notations and main results}
Throughout the whole paper, $M$ will be a compact manifold without
boundary.  Along the flow equation \eqref{metric evolving} the
Riemannian volume $dy$ of $M$ evolves by
$$
\frac{\partial}{\partial t} \,dy=-A\,dy
$$
and $A$ satisfies
$$
\frac{\partial A}{\partial t}=2|\alpha|^2+g^{ij}\frac{\partial
  \alpha_{ij}}{\partial t}
$$
where $|\alpha|^2=g^{ij}g^{kl}\alpha_{ik}\alpha_{jl}$. To simplify
notations, we let $\beta_{ij}:= \frac{\partial \alpha_{ij}}{\partial
  t}$ and $B:= g^{ij}\beta_{ij}$ so that
\begin{align}\label{A and B}
  \frac{\partial A}{\partial t}=2|\alpha|^2+B.
\end{align}
In particular, $A=R$ and $B=\Delta R$ under the Ricci flow.

For any time-dependent vector field $V$ on $M$ we define
\begin{align}\label{Theta definition}
  \Theta_{g,\alpha}\left(V\right):=\left({\Rc}-\alpha\right)(V,V)
  +\langle\nabla A-2\Div(\alpha),V\rangle+\frac 1 2\left(B-\Delta
    A\right)
\end{align}
where $\Rc$ is the Ricci tensor and $\Div$ the divergence operator,
i.e. $\Div(\alpha)_k=g^{ij}\nabla_i\alpha_{jk}$.  In the rest of this
paper we omit the subscripts of $\Theta_{g,\alpha}\left(V\right)$ and
denote it by $\Theta(V)$.

The quantity $\Theta(V)$ appears as an error term in our main results.
In the expression of $\Theta(V)$, the $\Rc$ term is caused by the
Bochner's formula. This explains technically why our results are
particularly useful for the Ricci flow and its various modifications.
In \cite{M} M\"{u}ller introduced the quantity $\mathcal D$.  In our
notations M\"{u}ller's definition reads as
$$
\mathcal D(V)=\partial_t A-\Delta
A-2|\alpha_{ij}|^2+4\nabla_i\alpha_{ij}V_j-2\nabla_j A
V_j+2R_{ij}V_iV_j-2\alpha_{ij}V_iV_j.
$$
Note that $\mathcal D$ and $\Theta$ are essentially the same; indeed
$\mathcal D(V)=2\Theta(-V)$. M\"uller \cite{M} further explained that
$\mathcal D$ is the difference between two differential Harnack type
quantities for the tensor $\alpha$.

Let $u(t, y)$ be a nonnegative solution to the conjugate heat equation
\begin{equation}\label{general conjugate heat equation}
  \frac{\partial u(t,y)}{\partial t}=-\Delta u(t,y)+A(t,y)\,u(t,y),
  \quad t\in(0, T),\ y\in M,
\end{equation}
where $\Delta$ is the Laplace-Beltrami operator calculated with
respect to the evolving metric $g(t)$.  Note that $\int_M u(t, y)\,dy$
remains constant along the flow, and without loss of generality we
assume this constant to be 1.

The classical Boltzmann-Shannon entropy functional is defined by
\begin{equation}\label{B-S entropy}
  \E(t)=\int_M u(t,y)\log u(t,y) \,dy.
\end{equation}

If $\Theta(V)\ge 0$ for all $V$, we will show that $\E$ is
convex. Based on this observation we construct Perelman's $\F$ and
$\W$ entropy in abstract geometric flows. We then derive the explicit
evolution equations of the entropies along the conjugate heat
equation, and show that they are monotone if $\Theta\ge 0$.  We thus
present a unified formula of various $\W$ entropies established by
various authors for different flows (including the static case), see
\cite{F-I-N, Li, L, M2, N, P}.

We show indeed that the generalized entropy $\F_k$ $(k\ge1)$, see
Definition \ref{Definition_F-definition} below, is monotone under the
additional assumption $B-\Delta A\ge0$, which is satisfied by all
previously mentioned flows. The study of the $\F_k$ entropy leads to a
simpler argument to rule out nontrivial steady breathers.

The eigenvalues and eigenfunctions of the Laplace-Beltrami operator
with potential $cA$ where $c$ is a constant, satisfy
\begin{equation}\label{eigenvalue definition}
  \lambda(t)f(t,y)=-\Delta f(t,y)+cA(t,y)f(t,y).
\end{equation}

Let $\lambda(t)$ be the lowest eigenvalue. We shall determine the
derivative of $\lambda(t)$.  A remarkable fact is that the derivative
$\lambda'(t)$ does not depend on the time derivative of the
corresponding eigenfunction; this allows to establish a formula for
$\lambda'(t)$ not requiring knowledge of the eigenfunction evolution.
We will prove eigenvalue monotonicity and apply it to rule out
nontrivial steady and expanding breathers in various flows.

\section{The first two derivatives of Boltzmann-Shannon entropy}

In this section we calculate the first two derivatives of the
Boltzmann-Shannon entropy.
\begin{theorem}\label{B-S entropy derivatives}
  Suppose that $(M, g(t))$ is a solution to the abstract geometric
  flow \eqref{metric evolving}, and that $u(t,y)$ is a positive
  solution to the conjugate heat equation \eqref{general conjugate
    heat equation}, normalized by $\int_M u(t,y) \,dy=1$. Then the
  first two derivatives of $\E(t)$ are given by
  \begin{equation}\label{1st derivative of B-S}
    \E'(t)=\int_M(|\nabla\log u|^2+A)u \,dy
  \end{equation}
  and
  \begin{align}\label{2nd derivative of B-S}
    \E''(t)=\int_M2\left(|\alpha-\nabla\nabla\log
      u|^2+\Theta(\nabla\log u)\right)u \,dy.
  \end{align}
  In particular, if $\Theta$ is nonnegative then $\E(t)$ is convex in
  time.
\end{theorem}

\begin{proof}
  Since $M$ is closed we can integrate by parts freely. Direct
  calculations show that
  \begin{align*}
    \E'(t)&=\int_M\left( u_t\log u+u_t-Au\log u \right)\,dy\\
    &=\int_M\left(-\Delta u+Au\right)\log u-\Delta u+Au-Au\log u \,dy\\
    &=\int_M\left(-\Delta u\log u+Au\right)\,dy\\
    &=\int_M\left(\left|\nabla\log u\right|^2+A\right)u\,dy
  \end{align*}
  and
  \begin{align*}
    \E''(t)=\int_M&\frac{\partial(|\nabla\log u|^2+A)}{\partial t}u+(|\nabla\log u|^2+A)\frac{\partial u}{\partial t}-(|\nabla\log u|^2+A)uA \,dy\\
    =\int_M&\left(2\alpha(\nabla\log u,\nabla\log u)+2\langle \nabla\frac{u_t}{u},\nabla\log u\rangle +2|\alpha|^2+B\right)u\\
    &+\left(\left|\nabla\log u\right|^2+A\right)\left(-\Delta u+Au\right)-\left(\left|\nabla\log u\right|^2+A\right)uA \,dy\\
    =\int_M&\left(2\alpha(\nabla\log u,\nabla\log u)+2\langle \nabla\left(-\frac{\Delta u}{u}+A\right),\nabla\log u\rangle \right)u\\
    &+\left(2|\alpha|^2+B\right)u-\left(\left|\nabla\log u\right|^2+A\right)\Delta u \,dy\\
    =\int_M&2u\alpha(\nabla\log u,\nabla\log u)-2\langle\nabla\left(\frac{\Delta u}{u}\right),\nabla u\rangle+2\langle\nabla A,\nabla u\rangle\\
    &+2u|\alpha|^2+Bu-\Delta(|\nabla\log u|^2) u-\Delta A u \,dy.
  \end{align*}
  Plugging in $\Delta\log u=\frac{\Delta u}{u}-|\nabla\log u|^2$ and
  \begin{align*}
    \Delta(|\nabla\log u|^2)=2|\nabla\nabla\log u|^2&+2\Rc(\nabla\log
    u, \nabla\log u)\\ &+2\langle \nabla\log u,\nabla(\Delta\log
    u)\rangle,
  \end{align*}
  we have
  \begin{align*}
    \E''(t)
    =\int_M&2u\left(|\alpha|^2+|\nabla\nabla\log u|^2\right)+2u\left(\alpha+\Rc\right)(\nabla\log u,\nabla\log u)\\
    &+Bu-3\Delta A u \,dy\\
    =\int_M&2u|\alpha-\nabla\nabla\log u|^2+4u\langle \alpha,\nabla\nabla\log u\rangle\\
    &+2u\left(\alpha+\Rc\right)(\nabla\log u,\nabla\log u)+(B-\Delta
    A)u+2\langle\nabla A,\nabla u\rangle \,dy.
  \end{align*}
  By observing that
$$
\Div \left(u\alpha(\nabla\log u)\right) =\alpha(\nabla\log u,\nabla
u)+u\Div(\alpha)(\nabla\log u)+u\langle\alpha,\nabla\nabla\log
u\rangle
$$
and by the divergence theorem, we get
\begin{align*}
  \E''(t)=\int_M&2u|\alpha-\nabla\nabla\log u|^2+2u\left({\Rc}-\alpha\right)(\nabla\log u,\nabla\log u)\\
  &+(B-\Delta A)u+\langle2\nabla A-4\Div(\alpha),\nabla u\rangle \,dy
\end{align*}
which is exactly \eqref{2nd derivative of B-S}.
\end{proof}

\section{Examples where $\Theta$ and $B-\Delta A$ are nonnegative}

In the following we list some examples where $\Theta$ and $B-\Delta A$
are nonnegative.  Calculations on the Ricci flow and extended Ricci
flow are carried out in details.  For other examples, we list values
of $\Theta$ and $B-\Delta A$ and for details we refer to M\"uller's
paper \cite{M}. This section is organized in the same way as the
corresponding section in \cite{M}.  Recall that
$$
\Theta\left(V\right)=\left({\Rc}-\alpha\right)(V,V)+\langle\nabla
A-2\Div(\alpha),V\rangle+\frac 1 2\left(B-\Delta A\right).
$$

\subsection{Riemannian manifold.}

In the case of a static metric we have $\alpha=0$ and hence
\begin{align}\label{Theta in Riemannian manifold}
  \Theta(V)=\Rc(V, V), \quad B-\Delta A=0.
\end{align}
Thus $\Theta$ is nonnegative if $M$ has nonnegative Ricci curvature.

\subsection{Hamilton's Ricci flow.}

In the case of Ricci flow where $\alpha=\Rc$, we have $A=R$. The
evolution equation $\frac{\partial R}{\partial t}=2|{\Rc}|^2+\Delta R$
gives
$$B=\frac{\partial A}{\partial t}-2|\alpha|^2=\Delta R.$$ Notice that
$\nabla R=2\Div({\Rc})$ by the second Bianchi identity, we thus get
\begin{align}\label{Theta in Ricci flow}
  \Theta(V)=0, \quad B-\Delta A=0.
\end{align}

\subsection{List's extended Ricci flow.}

In \cite{L} Bernhard List introduced an extended Ricci flow system,
namely
\begin{align}\label{extended Ricci flow equation}
  \frac{\partial g}{\partial t}={-2\Rc}+2a_n \,\nabla v\otimes \nabla
  v
\end{align}
where $v$ is a solution to the heat equation $\frac{\partial
  v}{\partial t}=\Delta v$ and $a_n$ a positive constant depending
only on the dimension $n$ of the manifold.  It turns out that one can
exhibit List's flow as a Ricci-DeTurck flow in one higher dimension.
This connection has been observed by Jun-Fang Li according to
\cite{AW}.  The extended Ricci flow is interesting by itself since its
stationary points are solutions to the vacuum Einstein equations, and
it is desirable to work on this flow directly.

In our notations for the extended Ricci flow,
$\alpha={\Rc}-a_n\,dv\otimes dv$ and $A=R-a_n|\nabla v|^2$, which
gives
$$\nabla A=\nabla R-2a_n\nabla\nabla v(\nabla v, \cdot).$$
Since $\Div(dv\otimes dv)_k=g^{ij}\nabla_i(\nabla_jv
\nabla_kv)=(\Delta v)\nabla_kv+g^{ij}\nabla_jv\nabla_i\nabla_kv$ we
have
$$
\Div\alpha=\Div{\Rc}-a_n\Div(dv\otimes dv)=\frac 1 2 \nabla
R-a_n\left(\Delta v\nabla v+\nabla\nabla v(\nabla v,\cdot)\right).
$$
Thus we find
\begin{equation}\label{nabla A - 2 Div alpha in extended Ricci flow}
  \nabla A-2\Div(\alpha)=2a_n\Delta v\nabla v.
\end{equation}
The evolution equation of $\alpha$ is given by (cf.~\cite{L})
$$
\beta_{ij}=\frac{\partial \alpha_{ij}}{\partial t}=\Delta
\alpha_{ij}-R_{ip}\alpha_{pj}-R_{jp}\alpha_{pi}+2R_{ipqj}\alpha_{pq}+2a_n\,\Delta
v\nabla_i \nabla_j v.
$$
(Note that by our notation $R_{ij}=g^{pq}R_{ipqj}$, while many authors
including List write $R_{ij}=-g^{pq}R_{ipqj}$.) Hence we have
$B=\Delta A+2a_n(\Delta v)^2$ and
\begin{align}\label{B - Delta A in extended Ricci flow}
  B-\Delta A=2a_n(\Delta v)^2.
\end{align}
Plugging in our formula of $\Theta$ we arrive at
\begin{align*}
  \Theta(V) &=a_n\langle\nabla v,V\rangle^2+2a_n\,\Delta v\langle\nabla
  v,V\rangle+a_n(\Delta v)^2\\
  &=a_n\left(\langle\nabla v,V\rangle+\Delta
    v\right)^2.
\end{align*}

\subsection{M\"uller's Ricci flow coupled with harmonic map flow}
The Ricci flow coupled with an harmonic map flow was introduced by
M\"uller in \cite{M2} as a generalization of the extended Ricci flow.
Suppose that $(N,\gamma)$ is a further closed static Riemannian
manifold, $a(t)$ a nonnegative function depending only on time, and
$\varphi(t)\colon M\to N$ a family of $1$-parameter smooth maps. Then
$(g(t),\varphi(t))$ is called a solution to M\"uller's Ricci flow
coupled with harmonic map flow with coupling function $a(t)$, if it
satisfies
\begin{equation}\label{RH flow equation}
  \left\{
    \begin{aligned}  &\frac{\partial g}{\partial t}=-2{\Rc}+2a(t)\,\nabla\varphi\otimes\nabla\varphi\\
      &\frac{\partial
        \varphi}{\partial t}=\tau_g\varphi
    \end{aligned}
  \right.
\end{equation}
where $\tau_g$ denotes the tension field of the map $\varphi$ with
respect to the evolving metric~$g(t)$ and
$\nabla\varphi\otimes\nabla\varphi\equiv\varphi^*\gamma$ the pullback
of the metric $\gamma$ on $N$ via the map~$\varphi$.

Recall that $\mathcal D(V)=2\Theta(-V)$; we have (cf.~\cite{M})
\begin{align}\label{Theta in RH flow}
  B-\Delta A= 2a\,|\tau_g\varphi|^2-a'|\nabla\varphi|^2,\quad
  \Theta(V)=a\,|\tau_g\varphi+\nabla_V\varphi|^2-\frac{a'}{2}|\nabla\varphi|^2.
\end{align}
Thus both $B-\Delta A$ and $\Theta$ are nonnegative as long as $a(t)$
is non-increasing in time.

\subsection{Lorentzian mean curvature flow when the ambient space has nonnegative sectional curvature}
Let $L^{n+1}$ be a Lorentzian manifold, and $M(t)$ be a family of
space-like hypersurfaces of $L$.  Denote by $\nu$ the future-oriented
time-like unit normal vector of $M$, by $h_{ij}$ the second
fundamental form, and by $H$ its mean curvature. Let $F(t, y)$ be the
position function of $M$ in $L$. The Lorentzian mean curvature flow is
then defined by
\begin{align}\label{Lorentzian mean curvature flow equation}
  \frac{\partial F}{\partial t}=H\nu.
\end{align}
The induced metric $g(t)$ of $M(t)$ satisfies $\partial_t
g=2Hh_{ij}$. We have
\begin{align}\label{Theta in Lorentzian mean curvature flow}
  &B-\Delta A=2H^2|h|^2+2|\nabla H|^2+2H^2\,\overline{\Rc}(\nu,\nu),\\
  &\Theta(V)=|\nabla H+
  h(V,\cdot)|^2+\overline{\Rc}(H\nu+V,H\nu+V)+\overline{\operatorname{Rm}}(V,\nu,\nu,V)\nonumber
\end{align}
where $\overline{\Rc}$ and $\overline{\operatorname{Rm}}$ denote the
Ricci, resp.~Riemann curvature tensor of $L^{n+1}$. Obviously both
$B-\Delta A$ and $\Theta$ are nonnegative when the sectional curvature
of $L^{n+1}$ is nonnegative.

\section{Perelman's $\F_k$ functional in abstract geometric flows}

We proved the following. If $(M, g(t))$ is a solution to the abstract
flow equation \eqref{metric evolving} and $u$ a positive solution to
the conjugate heat equation \eqref{general conjugate heat equation}
then
\begin{align}\label{derivative of F-functional -1}
  \frac{d}{dt}\int_M(|\nabla\log u|^2+A)u
  \,dy=\int_M2\left(|\alpha-\nabla\nabla\log u|^2+\Theta(\nabla\log
    u)\right)u \,dy.
\end{align}
We note that
\begin{align}\label{derivative of integration of Au}
  \frac{d}{dt}\int_MAu\,dy&=\int_M\frac{\partial A}{\partial
    t}u+A\frac{\partial u}{\partial t}-A^2u\,dy\\\nonumber
  &=\int_M\left(2|\alpha|^2+B\right)u+A\left(-\Delta
    u+Au\right)-A^2u\,dy\\\nonumber &=\int_M2\left(|\alpha|^2+\frac 1
    2\left(B-\Delta A\right)\right)u \,dy.
\end{align}
Let $\phi:=-\log u$ then
\begin{align}\label{phi equation}
  \frac{\partial\phi}{\partial t}=-\Delta\phi+|\nabla\phi|^2-A
\end{align}
with constraint $\int_M e^{-\phi}\,dy=1$. We rewrite
Eq.~\eqref{derivative of F-functional -1} in the more familiar form
following Perelman's notations:
\begin{align}\label{F 1 derivative}
  \frac{d}{dt}\int_M(|\nabla\phi|^2+A)e^{-\phi}
  \,dy=\int_M2\left(|\alpha+\nabla\nabla\phi|^2
    +\Theta(-\nabla\phi)\right)e^{-\phi}\,dy.
\end{align}

\begin{definition}\label{Definition_F-definition}
  For any $\phi\in C^\infty(M)$ with $\int_M e^{-\phi}\,dy=1$ and any
  constant $k$ we define Perelman's $\F_k$-functional for abstract
  geometric flows by
  \begin{align}\label{F-definition}
    \F_k(g,\phi)=\int_M\left(|\nabla\phi|^2+kA\right)e^{-\phi} \,dy.
  \end{align}
  When $k=1$ we simply denote $\F_1$ by $\F$.
\end{definition}

For Perelman's $\F_k$-functional in an abstract geometric flow we have
the following.
\begin{theorem}\label{theorem of F monotonicity}
  If $g$ is a solution of the abstract geometric flow \eqref{metric
    evolving} and $\phi$ a solution to Eq.~\eqref{phi equation} then
  we have
  \begin{align}\label{F k derivative}
    \frac{d}{dt}\F_k=&\int_M 2 \left(|\alpha+\nabla\nabla\phi|^2 +
      (k-1)|\alpha|^2 \right) e^{-\phi}\\ \nonumber
    &+2\left(\Theta(-\nabla\phi)+\frac{k-1}{2}\left(B-\Delta
        A\right)\right)e^{-\phi} \,dy.
  \end{align}
  Thus for $k>1$, $\F_k$ is monotone nondecreasing as long as
  $B-\Delta A$ and $\Theta$ are nonnegative.  Moreover the
  monotonicity is strict unless
  \begin{align*}
    \alpha=0, \quad \phi=\operatorname{constant}, \quad B-\Delta A=0.
  \end{align*}
  For $k=1$ we have
  \begin{align}\label{F derivative}
    \frac{d}{dt}\F=\int_M 2 \left(|\alpha+\nabla\nabla\phi|^2
      +\Theta(-\nabla\phi) \right) e^{-\phi}\,dy.
  \end{align}
  In particular, $\F$ is monotone nondecreasing when $\Theta\ge 0$,
  and the monotonicity is strict unless
  \begin{align*}
    \alpha+\nabla\nabla\phi=0, \quad \Theta(-\nabla\phi)=0.
  \end{align*}
\end{theorem}
\begin{proof}
  Since
  \begin{align*}
    \F_k(g,\phi)=\int_M(|\nabla\phi|^2+A)e^{-\phi} \,dy+(k-1)\int_M A
    e^{-\phi}\,dy,
  \end{align*}
  and by Eqs.~\eqref{F 1 derivative} and \eqref{derivative of
    integration of Au} we immediately get formula~\eqref{F k
    derivative}.

  Furthermore for $k>1$, the functional $\F_k$ is monotone
  nondecreasing as long as $B-\Delta A$ and $\Theta$ are nonnegative.
  When $\frac{d}{dt}\F_k=0$, each term on the RHS of Eq.~\eqref{F k
    derivative} has to be identically zero. In particular we have
$$\alpha+\nabla\nabla\phi=0,\quad \alpha =0$$
which further implies $\Delta\phi=0$ on the closed manifold $M$, and
thus $\phi$ has to be a constant.  Now
$\Theta(-\nabla\phi)=\Theta(0)=(B-\Delta A)/2$ and $B-\Delta A=0$.

When $k=1$ the statement in the theorem is obvious.
\end{proof}

The advantage of $\F_k$ over $\F$ is that when $k>1$, extra terms in
$\F'_k$ can tell more about the manifold $M$. Li \cite{Li} has studied
$\F_k$ in the Ricci flow.  We state an analogous application of $\F_k$
to rule out nontrivial steady breathers in abstract geometric flows.

Recall that a breather of a geometric flow is a periodic solution
changing only by diffeomorphism and rescaling.  A solution $(M, g(t))$
is called a breather if there are a diffeomorphism $\eta\colon M\to
M$, a positive constant $c$ and times $t_1<t_2$ such that
\begin{align}\label{breather definition for RH flow}
  g(t_2)=c\,\eta^* g(t_1) ,\quad \alpha(t_2)=\eta^*\alpha(t_1).
\end{align}
When $c<1$, $c=1$ or $c>1$, the breather is called shrinking, steady
or expanding, respectively.

We now apply monotonicity of $\F_k$ to rule out nontrivial steady
breathers of abstract geometric flows.
\begin{corollary}\label{corrollary no steady breather in RH flow}
  Suppose that $(M, g(t))$ is a steady breather to an abstract
  geometric flow~\eqref{metric evolving}.  Suppose that $\Theta\ge0$
  and $B-\Delta A\ge0$. Then $B-\Delta A=0$ and the steady breather is
  $\alpha$-flat, namely $\alpha=0$.
\end{corollary}

\begin{proof} The arguments are standard and follow from Perelman's
  proof of the no steady breather theorem for the Ricci flow \cite{P}.
  We follow the notes by Kleiner and Lott \cite{K-L} and only sketch
  the proof.  Define
  \begin{align}\label{lambda definition}
    \lambda(t)=\inf\left\{\F_k(g,\phi)\colon\ \int_M e^{-\phi}
      \,dy=1,\ \phi\in C^\infty(M)\right\}.
  \end{align}
  Since we are on a steady breather we have
  $\lambda(t_1)=\lambda(t_2)$. Let $\bar \phi(t_2)$ be a minimizer of
  $\lambda(t_2)$. Solve the conjugate heat equation backwards with end
  value $e^{-\bar \phi(t_2)}$. Denote the solution by $u(t)$.  Let
  $\phi(t)=-\log u(t)$ then $\phi(t)$ satisfies the constraint
$$\int_M e^{-\phi} \,dy=1,$$
and $\F_k(g(t), \phi(t))$ is monotone nondecreasing as its derivative
is nonnegative when $e^{-\phi(t)}$ is a solution to the conjugate heat
equation.  Thus we have
\begin{equation}\label{eigenvalue monotonicity inequality}
  \lambda(t_1)\le \F_k(g(t_1), \phi(t_1))\le\F_k(g(t_2),\bar \phi(t_2))=\lambda(t_2).
\end{equation}
Since on a breather $\lambda(t_1)=\lambda(t_2)$, we get
$$
\mathcal F_k(g(t_1), \phi(t_1))=\mathcal F_k(g(t_2), \phi(t_2)),
$$
and in particular $\F'_k\left(g(t),\phi(t)\right)=0$ when
$t\in[t_1,t_2]$. Now we apply Theorem~\ref{theorem of F monotonicity}
to conclude that $\alpha=0$ and $B-\Delta A=0$ on $M$ when
$t\in[t_1,t_2]$.
\end{proof}

\begin{remark}\label{eigenvalue monotonicity remark}
  From Eq.~\eqref{lambda definition} we know that $\lambda$ is the
  lowest eigenvalue of $-\Delta+\frac k 4 A$.  Thus, by
  Theorem~\textup{\ref{theorem of F monotonicity}}, under the
  assumptions that $B-\Delta A\ge 0$ and $\Theta\ge0$, the lowest
  eigenvalue of $-\Delta+\frac k 4 A$ is monotone in $t$ when $k\ge
  1$.  An explicit formula for the derivative of the lowest eigenvalue
  will be given in Sect.~$7$ under the assumption that $\lambda$ is
  differentiable along time.
\end{remark}
\section{Construction of Perelman's $\W$ entropy}
We have noted that Perelman's $\F$-functional is the derivative of
$\E$, whose stationary points are steady solitons.  The purpose of
this section is to construct functionals corresponding to the
shrinking solitons.  Our construction is just completing squares of
$\E''$ (or $\F'$ by Perelman's notation). Monotonicity of $\W$ holds
in the flows mentioned in Section 3.

We rewrite the second derivative of $\E(t)$ in order to fit the
shrinking soliton equation simply by completing squares.
\begin{align*}
  \E''(t)=\int_M&2\left(|\alpha-\nabla\nabla\log u|^2+\Theta(\nabla\log u)\right)u \,dy\\
  =\int_M&2u\left|\alpha-\nabla\nabla\log u-\frac 1{2(T-t)}g\right|^2+\frac{2u}{T-t}(A-\Delta\log u)\\
  &-\frac{2nu}{4(T-t)^2}+2u\Theta(\nabla\log u) \,dy\\
  =\int_M&2\left(\left|\alpha-\nabla\nabla\log u-\frac 1{2(T-t)}g\right|^2+\Theta(\nabla\log u)\right)u\,dy\\
  &+\frac 2{T-t}\E'(t)-\frac n{2(T-t)^2}\,.
\end{align*}
Hence we have
\begin{align*}
  \int_M&2\left(\left|\alpha-\nabla\nabla\log u-\frac 1{2(T-t)}g\right|^2+\Theta(\nabla\log u)\right)u\,dy\\
  &=\E''(t)-\frac 2{T-t}\E'(t)+\frac n{2(T-t)^2}\\
  &=\frac{1}{T-t}\frac{d}{dt}\left((T-t)\E'-\E-\frac n 2
    \log(T-t)\right).
\end{align*}
Now in terms of
\begin{align*}
  \W := (T-t) \E'-\E-\frac n 2 \log(T-t)-\frac n 2 \log(4\pi)-n,
\end{align*}
we proved that
\begin{equation}\label{W derivative in terms of u}
  \frac{d}{dt}\W=(T-t)\int_M2\left(\left|\alpha-\nabla\nabla\log u-\frac 1{2(T-t)}g\right|^2+\Theta(\nabla\log u)\right)u\,dy.
\end{equation}

Following Perelman, we let $$\tau:= T-t,\quad \phi:= -\log
\left(\left(4\pi\tau\right)^{n/2}u\right)$$ and introduce the
following definition.

\begin{definition}\label{W definition in abstract geometric flows}
  For a solution $(M, g)$ to the abstract geometric flow \eqref{metric
    evolving} and for $\phi\in C^\infty(M)$, let Perelman's
  $\W$-entropy be defined as
  \begin{align}\label{W in terms of phi}
    \W(g,\phi, t)=\int_M\left(\tau\left(\left|\nabla
          \phi\right|^2+A\right)+\phi-n\right)\left(4\pi\tau\right)^{-n/2}e^{-\phi}\,dy.
  \end{align}
\end{definition}

We can rewrite Eq.~\eqref{W derivative in terms of u} in the following
way.

\begin{theorem}\label{shrinking entropy monotonicity}
  Let $(M, g)$ be a solution to the abstract geometric
  flow~\eqref{metric evolving}. If $\phi$ satisfies
  \begin{align*}
    \frac{\partial \phi}{\partial t}=-\Delta \phi+|\nabla
    \phi|^2-A+\frac{n}{2\tau}
  \end{align*}
  such that
  \begin{align*}
    \int_M(4\pi\tau)^{-n/2}e^{-\phi}\,dy=1,
  \end{align*}
  then
  \begin{align*}
    \frac{d}{dt}\W=\int_M2\tau\left(\left|\alpha+\nabla\nabla
        \phi-\frac 1{2\tau}g\right|^2 +\Theta(-\nabla
      \phi)\right)(4\pi\tau)^{-n/2}e^{-\phi}\,dy.
  \end{align*}
  If $\Theta\ge 0$ then $\W$ is monotone nondecreasing, and the
  monotonicity is strict unless
  \begin{align*}
    \alpha+\nabla\nabla \phi-\frac 1{2\tau}g=0, \quad \Theta(-\nabla
    \phi)=0.
  \end{align*}
\end{theorem}

The monotonicity of $\W$ can be applied to rule out nontrivial
shrinking breathers in abstract flows with $\Theta\ge 0$.  The
arguments are almost identical to the Ricci flow case. We omit
details.

\begin{remark}
Monotonicity of $\W$ was previously proven by Hong Huang in \cite{H}. The advantage of Theorem \ref{shrinking entropy monotonicity} is that explicit formula of $\frac{d\W}{d t}$ is given. We thank Professor Huang for bringing \cite{H} to our attention.
\end{remark}
\section{Expander entropy $\W_+$}
Feldman-Ilmanen-Ni \cite{F-I-N} established expander entropy $\W_+$
for Ricci flow, and there has been a very nice explanation of their
motivation in \cite{F-I-N}. We attempt to explain formally why $\W_+$
should be the way as they defined it. In short, the signs in $\W$ and
$\W_+$ are caused by antiderivatives of $1/(t-T)$ depending on the
situation whether $t>T$ or $t<T$.

We now carry out the details. Note that on expanders $t>T$ and that
\begin{align*}
  \E''(t)=\int_M&2\left(|\alpha-\nabla\nabla\log u|^2+\Theta(\nabla\log u)\right)u \,dy\\
  =\int_M&2u\left|\alpha-\nabla\nabla\log u+\frac 1{2(t-T)}g\right|^2-\frac{2u}{t-T}(A-\Delta\log u)\\
  &-\frac{2nu}{4(t-T)^2}+2u\Theta(\nabla\log u) \,dy\\
  =\int_M&2\left(\left|\alpha-\nabla\nabla\log u+\frac 1{2(t-T)}g\right|^2+\Theta(\nabla\log u)\right)u\,dy\\
  &-\frac 2{t-T}\E'(t)-\frac n{2(t-T)^2},
\end{align*}
moreover
\begin{align*}
  \int_M2&\left(\left|\alpha-\nabla\nabla\log u+\frac 1{2(t-T)}g\right|^2+\Theta(\nabla\log u)\right)u\,dy\\
  &\quad=\E''(t)+\frac 2{t-T}\E'(t)+\frac n{2(t-T)^2}\\
  &\quad=\frac{1}{t-T}\frac{d}{dt}\left((t-T)\E'+\E+\frac n 2
    \log(t-T)\right).
\end{align*}
The calculations suggest to define
$$\W_+:=(t-T)\E'+\E+\frac n 2 \log(t-T)+\frac n 2 \log(4\pi)+n$$
which is the definition of expander entropy in \cite{F-I-N} in the
case of Ricci flow.  One has
$$
\frac{d\W_+}{dt}=(t-T)\int_M2\left(\left|\alpha-\nabla\nabla\log
    u+\frac 1{2(t-T)}g\right|^2+\Theta(\nabla\log u)\right)u\,dy.
$$
This again may be rewritten following \cite{F-I-N} in terms of
$$\sigma:= t-T,\quad
\phi_+ := -\log\left(\left(4\pi\sigma\right)^{n/2}u\right).$$

\begin{definition}\label{W for expander definition in abstract geometric flows}
  For a solution $(M, g)$ to the abstract geometric flow \eqref{metric
    evolving} and $\phi_+ \in C^\infty(M)$ one defines Perelman's
  entropy for expanders by
  \begin{align}\label{W expander in terms of phi}
    \W_+(g,\phi_+, t)=\int_M\left(\sigma\,(|\nabla
      \phi_+|^2+A)-\phi_++n\right)(4\pi\sigma)^{-n/2}e^{-\phi_+}\,dy.
  \end{align}
\end{definition}

\begin{theorem}\label{expander entropy monotonicity}
  Let $(M, g)$ be a solution to the abstract geometric
  flow~\eqref{metric evolving}. Assume that $\phi_+$ satisfies
  \begin{align*}
    \frac{\partial \phi_+}{\partial t}=-\Delta \phi_++|\nabla
    \phi_+|^2-A-\frac{n}{2(t-T)}
  \end{align*}
  such that
  \begin{align*}
    \int_M(4\pi\sigma)^{-n/2}e^{-\phi_+}\,dy=1.
  \end{align*}
  We have
  \begin{equation*}
    \frac{d\W_+}{dt}=\int_M2\sigma\left(\left|\alpha+\nabla\nabla \phi_++\frac 1{2(t-T)}g\right|^2
      +\Theta(-\nabla \phi_+)\right)(4\pi\sigma)^{-n/2}e^{-\phi_+}\,dy.
  \end{equation*}
  Furthermore, if $\Theta\ge 0$ then $\W_+$ is monotone nondecreasing,
  and the monotonicity is strict unless
  \begin{align*}
    \alpha+\nabla\nabla \phi_++\frac 1{2(t-T)}g=0, \quad
    \Theta(-\nabla \phi_+)=0.
  \end{align*}
\end{theorem}

\begin{remark}
  The constants $\pm\left( \frac n 2 \log(4\pi)+n\right)$ in the
  definition of $\W$ and $\W_+$ are for purposes of normalization.
\end{remark}

\section{Evolution equation of the lowest eigenvalue}
In this section, assuming that the lowest eigenvalue $\lambda(t)$ is
differentiable along ~$t$, we derive an explicit formula for its
derivative 
in terms of its normalized eigenfunction.  Although monotonicity of
$\F_k$ in Theorem \ref{theorem of F monotonicity} is sufficient for
our geometric applications, an explicit formula which holds at points
where $\lambda$ is differentiable, may be of independent interest.
Time derivatives of the eigenfunction do not appear in the formula.

In the literature, for instance \cite[Section~7]{K-L}, it has been
stated that smooth dependence on time of the lowest eigenvalue and the
corresponding eigenfunction follows from perturbation theory as
presented in the book by Reed and
Simon~\cite[Chapt.~XII]{reedsimon}. However it is not immediately
clear how perturbation theory is applied to our context, where the
operator depends only smoothly, but not analytically on~$t$.

\begin{lemma}\label{lemma_of_integral_of_psi_Deltaf^2}
  Assume that $M$ is a closed manifold and let $\psi\in C^\infty(M)$.
  Let $\lambda$ be the lowest eigenvalue of $-\Delta +\psi$ and $f$ a
  positive eigenfunction corresponding to $\lambda$, i.e.  $\lambda
  f=-\Delta f+\psi f$. Then
  \begin{align}\label{integral of psi Delta f ^2}
    \int_M\psi \Delta f^2\,dy=\int_M 2 \left(\left|\nabla\nabla \log
        f\right|^2+\Rc(\nabla\log f, \nabla\log f)\right)f^2\,dy.
  \end{align}
\end{lemma}
\begin{proof}
  We have $\psi f=\lambda f+\Delta f$ and
  \begin{align*}
    \psi \Delta f^2&=2\psi f\Delta f+2\psi |\nabla f|^2\\
    &=2(\lambda f+\Delta f)\Delta f+2(\lambda f+\Delta f)\frac{|\nabla f|^2}{f}\\
    &=\lambda\left(2f\Delta f+2|\nabla f|^2\right)+2\left({\Delta f}\right)^2+2\,\frac{\Delta f |\nabla f|^2}{f}\\
    &=\lambda\Delta f^2+2\left({\Delta f}\right)^2+2\,\frac{\Delta f
      |\nabla f|^2}{f}.
  \end{align*}
  We observe that
  \begin{align}\label{psi Delta f formula}
    \int_M\psi \Delta f^2\,dy&=\int_M2\left({\Delta
        f}\right)^2+2\,\frac{\Delta f |\nabla f|^2}{f}\,dy\\\nonumber
    &=\int_M-2\,\langle\nabla f,\nabla(\Delta
    f)\rangle-2\,\langle\nabla f,\nabla\left(\frac{|\nabla
        f|^2}{f}\right)\rangle \,dy.
  \end{align}
  Now we calculate the two terms of the RHS in Eq.~\eqref{psi Delta f
    formula}. For the first term we have by Bochner's formula
  \begin{align}\label{1st term in psi Delta f formula}
    -2\langle \nabla f,\nabla(\Delta f)\rangle=2|\nabla\nabla
    f|^2+2\Rc(\nabla f, \nabla f)-\Delta(|\nabla f|^2).
  \end{align}
  The second term writes as
  \begin{align}\label{2nd term in psi Delta f formula}
    \langle\nabla f,\nabla\left(\frac{|\nabla
        f|^2}{f}\right)\rangle&=\langle\nabla
    f,\nabla\left(f|\nabla\log f|^2\right)\rangle\\ \nonumber
    &=\langle\nabla f,\nabla f|\nabla\log f|^2+2f\nabla\nabla\log
    f(\nabla\log f,\cdot)\rangle\\ \nonumber &=f^2\left|\nabla\log
      f\right|^4+2f^2\nabla\nabla\log f\left(\nabla\log f,\nabla\log
      f\right)\\ \nonumber &=\left|\nabla\nabla
      f\right|^2-f^2\left|\nabla\nabla \log f\right|^2
  \end{align}
  where in the last equality we used that
  \begin{align*}
    \nabla\nabla\log f=\frac{\nabla\nabla f}{f}-\frac{\nabla
      f\otimes\nabla f}{f^2}=\frac{\nabla\nabla f}{f}-\nabla\log
    f\otimes\nabla\log f,
  \end{align*}
  and moreover
  \begin{align*}
    \left|\nabla\nabla f\right|^2&=f^2\left|\nabla\nabla\log f+\nabla\log f\otimes\nabla\log f\right|^2\\
    &=f^2\left|\nabla\nabla\log f\right|^2+2f^2\nabla\nabla\log
    f\left(\nabla\log f, \nabla\log f\right)+f^2\left|\nabla\log
      f\right|^4.
  \end{align*}
  Plugging Eqs.~\eqref{1st term in psi Delta f formula} and \eqref{2nd
    term in psi Delta f formula} into Eq.~\eqref{psi Delta f formula}
  we get
  \begin{equation*}
    \int_M\psi\Delta f^2\,dy=\int_M 2 f^2 \left|\nabla\nabla \log
      f\right|^2+2\Rc(\nabla f, \nabla f)\,dy.\qedhere
  \end{equation*}
\end{proof}

Let $\lambda(t)$ be the lowest eigenvalue of $-\Delta+cA$ where $c$ is
a constant, indeed
\begin{align}\label{eigenvalue definition - integral}
  \lambda(t)=\inf\left\{\int_M\left|\nabla\phi\right|^2+cA\phi^2 \,dy
    :\int_M\phi^2 \,dy=1,\ \phi\in C^\infty(M)\right\}.
\end{align}
Let $f(t,\cdot)$ be the corresponding positive eigenfunction
normalized by $$\int_M f^2(t,y)\,dy=1.$$
\begin{theorem}\label{theorem of eigenvalue monotonicity}
  At all times $t_0$ at which the function $t \mapsto \lambda(t)$ is
  differentiable we have
  \begin{align}\label{eigenvalue derivative}
    \lambda'(t_0) &= \frac{1}{2} \int_M \Big( \left| \alpha - 2 \nabla
      \nabla \log f\right|^2 + \left( 4 c - 1 \right) |\alpha|^2
    \\
    &\quad+ \Theta(2 \nabla \log f) + \frac{4c-1}{2} \left( B - \Delta
      A \right) \Big) f^2 \, dy.  \notag\end{align} In particular, for
  $c=1/4$ we have
  \begin{align}\label{eigenvalue derivative when c=1/4}
    \lambda' = \frac{1}{2} \int_M \left(\left|\alpha-2\nabla\nabla\log
        f\right|^2+\Theta(2\nabla\log f)\right)f^2 \,dy.
  \end{align}
\end{theorem}

\begin{proof}
  Fix $t_0 \in (0,T)$ where the function $t \mapsto \lambda(t)$ is
  differentiable, and let $\varphi: (0,T) \times M \to \R_{>0}$ be a
  smooth function such that
  \begin{enumerate}
  \item $\displaystyle\int_M \varphi(t,y)^2 dy = 1$ for all $t \in (0,T)$, and\\

  \item $\varphi(t_0, \cdot) = f(t_0, \cdot)$.
  \end{enumerate}
  For instance $\varphi(t)$ may be chosen as $f(t_0)\sqrt{dy(g(t_0))/d
    y(g(t))}$ where $dy(g(t))$ is the volume form with respect to
  metric $g(t)$.  Let
  \begin{equation}\label{mut}
    \mu(t) := \int_M \left( |\nabla \varphi(t,y)|^2 + c A(t,y) \varphi(t,y)^2 \right) dy.
  \end{equation}
  Then $\mu(t)$ is a smooth function by definition.  The trick to work
  with $\mu(t)$ rather than $\lambda(t)$ allows to bypass time
  derivatives of the eigenfunction $f(t, \cdot)$.  Note that $\mu(t)
  \geq \lambda(t)$ for all $t \in (0,T)$, and $\mu(t_0) =
  \lambda(t_0)$, so that $$\lambda'(t_0) = \mu'(t_0).$$

  Differentiation of \eqref{mut} gives
  \begin{align*}
    \mu' =\int_M & 2 \alpha(\nabla \varphi, \nabla \varphi) + 2 \langle \nabla \varphi', \nabla \varphi \rangle + c A' \varphi^2 + 2 c A \varphi \varphi'\\
    & -\left( |\nabla \varphi|^2 + c A \varphi^2 \right) A \, dy\\
    = \int_M & 2 \alpha( \nabla \varphi, \nabla \varphi) - 2 \varphi' \Delta \varphi + c A' \varphi^2 + 2 c A \varphi \varphi'\\
    & + \varphi \langle \nabla A, \nabla \varphi \rangle + A \varphi \Delta \varphi - c A^2 \varphi^2 \, dy\\
    = \int_M & 2 \alpha(\nabla \varphi, \nabla \varphi) + c A' \varphi^2 + \varphi \langle \nabla A, \nabla \varphi \rangle \, dy + \lambda \int_M 2 \varphi' \varphi - A\varphi^2 \, dy\\
    = \int_M & 2 \alpha( \nabla \varphi, \nabla \varphi) + c A' \varphi^2 + \varphi \langle \nabla A, \nabla \varphi \rangle \, dy\\
    = \int_M & 2 \alpha(\nabla \varphi, \nabla \varphi) + c \left( 2 |\alpha|^2 + B \right) \varphi^2 - \frac 1 2 A \Delta \varphi^2 \, dy\\
    = \int_M & 2 \alpha(\nabla \varphi, \nabla \varphi) + 2 c
    |\alpha|^2 \varphi^2 + c \left( B - \Delta A \right) \varphi^2 + c
    A \Delta \varphi^2 - \frac 1 2 A \Delta \varphi^2 \, dy,
  \end{align*}
  where in the fourth equality we used
  that $\int_M \left( 2 \varphi \varphi' - A \varphi^2 \right) dy = 0$
  (which is due to the normalization of $\varphi$).

  Noting that
  \begin{align*}
    \Div \left( \varphi \alpha(\nabla \varphi, \cdot) \right) & = \alpha( \nabla \varphi, \nabla \varphi) + \varphi \Div(\alpha)(\nabla \varphi) + \varphi \langle \alpha, \nabla \nabla \varphi \rangle\\
    & = 2 \alpha(\nabla \varphi, \nabla \varphi) + \varphi
    \Div(\alpha)(\nabla \varphi) + \varphi^2 \langle \alpha, \nabla
    \nabla \log \varphi \rangle
  \end{align*}
  and by the divergence theorem, we have
  \begin{align}\label{integral of alpha in lambda}
    & \int_M 2 \alpha( \nabla \varphi, \nabla \varphi) \, dy = \int_M
    4 \alpha( \nabla \varphi, \nabla \varphi) - 2 \alpha(\nabla
    \varphi, \nabla \varphi) \, dy\\ \nonumber & = \int_M -2 \varphi
    \Div(\alpha)(\nabla \varphi) - 2 \varphi^2 \langle \alpha, \nabla
    \nabla \log \varphi \rangle - 2 \alpha(\nabla \varphi, \nabla
    \varphi) \, dy.
  \end{align}
  In Eq.~\eqref{integral of psi Delta f ^2} let $\psi=cA$, then we get
  \begin{align}\label{integral of c A Delta f ^2}
    \int_McA \Delta \varphi^2\,dy=\int_M 2 \varphi^2
    \left|\nabla\nabla \log \varphi\right|^2+2\Rc(\nabla \varphi,
    \nabla \varphi)\,dy.
  \end{align}
  Plugging \eqref{integral of alpha in lambda} and \eqref{integral of
    c A Delta f ^2} into the equation for $\mu'$ we obtain
  \begin{align*}
    \mu' =\int_M&-2\varphi\Div(\alpha)(\nabla \varphi)-2\varphi^2\langle\alpha,\nabla\nabla \log \varphi\rangle-2\alpha(\nabla \varphi, \nabla \varphi)+2c|\alpha|^2\varphi^2\\
    &+c\left(B-\Delta A\right)\varphi^2+2 \varphi^2 \left|\nabla\nabla \log \varphi\right|^2+2\Rc(\nabla \varphi, \nabla \varphi)-\frac 1 2 A \Delta \varphi^2 \,dy\\
    =\int_M&\left(2\left|\nabla\nabla \log \varphi\right|^2-2\langle\alpha,\nabla\nabla \log \varphi\rangle+\frac 1 2|\alpha|^2+\left(2c-\frac 1 2\right)|\alpha|^2\right)\varphi^2\\
    &+\left(2\left({\Rc}-\alpha\right)\left(\nabla\log \varphi,\nabla\log \varphi\right)+\langle\nabla A-2\Div(\alpha),\nabla\log \varphi\rangle\right)\varphi^2\\
    &+c\left(B-\Delta A\right)\varphi^2\,dy\\
    =\int_M&\left(\frac 1 2\left|\alpha-2\nabla\nabla\log \varphi\right|^2+\left(2c- \frac 1 2 \right)|\alpha|^2\right)\varphi^2\\
    &+\left( \frac 1 2\Theta(2\nabla\log \varphi)+\left(c-\frac 1
        4\right) \left(B-\Delta A\right) \right)\varphi^2\,dy,
  \end{align*}
  so that
  \begin{align*}
    \lambda'(t_0) = \mu'(t_0) = \frac{1}{2} &\int_M \Big( \left| \alpha - 2 \nabla \nabla \log f\right|^2 + \left( 4 c - 1 \right) |\alpha|^2\\
    &+ \Theta(2 \nabla \log f) + \frac{4c-1}{2} \left( B - \Delta A
    \right) \Big) f^2 \, dy,
  \end{align*}
  as claimed.
\end{proof}

Let us compare Theorems \ref{theorem of F monotonicity} and
\ref{theorem of eigenvalue monotonicity}, resp.~formulas \eqref{F k
  derivative} and \eqref{eigenvalue derivative}.  Let $\phi=-2\log f$,
then Eq.~\eqref{eigenvalue derivative} can be rewritten as
\begin{align}\label{eigenvalue derivative in phi}
  \lambda' = \frac{1}{2} &\int_M \Big( \left| \alpha + \nabla \nabla \phi \right|^2 + (4c-1) |\alpha|^2\\
  &+ \Theta(-\nabla \phi) + \frac{4c-1}{2} \left( B - \Delta A\right)
  \Big) e^{-\phi} \, dy.\notag
\end{align}
Letting $k=4c$, we see that the two evolution equations are formally
proportional. We note that in Eq.~\eqref{F k derivative} the
exponential $e^{-\phi}$ is a normalized solution to the conjugate heat
equation, while $e^{-\phi/2}$ in \eqref{eigenvalue derivative in phi}
is the normalized eigenfunction of~$\lambda(t)$.

\section{Eigenvalue monotonicity in various flows}

In this section we list explicit formulas of the eigenvalue evolution
in different flows. The constant $c$ is assumed to be no less than
$1/4$.

\subsection{Hamilton's Ricci flow}
In the case of Ricci flow, monotonicity of the lowest eigenvalue of
$-\Delta+cR$ for $c\ge 1/4$ and its applications has been established
by Cao \cite{Cao1, Cao2} as mentioned in the introduction. See also
the work of Li \cite{Li}. Plugging
$$\alpha=\Rc, \quad\Theta=0, \quad B-\Delta A=0$$
into Eq.~\eqref{eigenvalue derivative} we get Cao's formula for the
Ricci flow \cite{Cao2}:
\begin{align*}
  \lambda'(t)=\int_M\frac 1 2\left(\left|{\Rc}-2\nabla\nabla\log
      f\right|^2+\left(4c- 1 \right)|{\Rc}|^2\right)f^2 \,dy.
\end{align*}
This can be applied to show that every steady breather in the Ricci
flow is Ricci flat.

\subsection{List's Extended Ricci flow}
We work out the details in the extended Ricci flow.
\begin{corollary}
  Assume that $(M, g(t))$ is a solution to the extended Ricci flow
  equation, and that $\lambda(t)$ is the lowest eigenvalue of
$$
-\Delta+c\left(R-a_n|\nabla v|^2\right),
$$
then we have
\begin{align}\label{lambda derivative in extended Ricci flow}
  \lambda'(t)=\int_M&\frac 1 2\left|{\Rc}-a_n\nabla v\otimes\nabla
    v-2\nabla\nabla\log f\right|^2 f^2\\\nonumber &+\left(2c-\frac 1
    2\right)\left|{\Rc}-a_n\nabla v\otimes\nabla v\right|^2
  f^2\\\nonumber &+\frac{a_n}2\left(\left(\Delta v-2\langle\nabla
      v,\nabla\log f\rangle\right)^2+\left(4 c- 1 \right) \left(\Delta
      v\right)^2\right)f^2dy.
\end{align}
In particular, a steady breather of the extended Ricci flow is trivial
in the sense that
\begin{align*}
  \Rc=0,\quad v\equiv \operatorname{const}.
\end{align*}
\end{corollary}
\begin{proof}
  Eq.~\eqref{lambda derivative in extended Ricci flow} is a direct
  plug-in. When $(M, g(t))$ is a steady breather, there are times
  $t_1<t_2$ such that $\lambda(t_1)=\lambda(t_2)$ for any $c>1/4$. In
  particular we have $\Delta v=0$ on the closed manifold $M$, thus $v$
  is constant, and moreover $M$ is Ricci flat by ${\Rc}-a_n\nabla
  v\otimes\nabla v=0$.
\end{proof}

\subsection{M\"uller's Ricci flow coupled with harmonic map flow}
We already used $\F_k$ to rule out nontrivial steady breathers. Using
eigenvalue monotonicity, one does not need to solve the conjugate heat
equation.  The lowest eigenvalue of
$$-\Delta+c\left(R-a(t)|\nabla\varphi|^2\right)$$
is nondecreasing along the flow. The conclusions remain the same as in
Corollary~\ref{corrollary no steady breather in RH flow}.

\subsection{Lorentzian mean curvature flow when the ambient space has nonnegative sectional curvature}
When $M$ evolves along the Lorentzian mean curvature flow
\eqref{Lorentzian mean curvature flow equation}, the lowest eigenvalue
of
$$-\Delta-cH^2$$
is nondecreasing provided sectional curvature of the ambient space is
nonnegative.

\section{Normalized eigenvalue and no expanding breathers theorem}
The eigenvalue of $-\Delta+cA$ is not scale invariant. Suppose that
$\alpha$ is invariant under scaling which is true in all of our
examples. If we re-scale a Riemannian metric $g$ to $\varepsilon g$ by
a positive constant $\varepsilon$, then
\begin{align*}
  -\Delta_{\varepsilon g}+cA_{\varepsilon
    g}=\varepsilon^{-1}\left(-\Delta_{g}+cA_{g}\right),
\end{align*}
and for the lowest eigenvalue we get $\lambda_{\varepsilon
  g}=\varepsilon^{-1}\lambda_g$. Thus the (non-normalized) lowest
eigenvalue only works in the steady case. Following Perelman \cite{P}
we define the scale invariant eigenvalue by
\begin{align}\label{definition of normalized eigenvalue}
  \bar\lambda_g:=\lambda_g V^{2/n}_g
\end{align}
where $V$ denotes the volume of $M$.

In the following for simplicity of calculations we let $c=1/4$.
\begin{proposition}\label{normalized eigenvalue increases whenever
    negative}
  Suppose that $(M, g(t))$ is a solution to the abstract geometric
  flow \eqref{metric evolving} with $\alpha$ being scale
  invariant. Assume that $\Theta$ is nonnegative. Let $\lambda(t)$ be
  the lowest eigenvalue of $-\Delta+A/4$.  Then whenever
  $\bar\lambda(t)\le 0$ one has $\bar\lambda'(t)\ge 0$.
\end{proposition}

\begin{proof}
  Recall that by Eq.(\ref{eigenvalue definition - integral}) and
  choosing $\phi(t,y)={V}^{-1/2}$ we
  have $$\lambda(t)\le\frac{1}{4V}\int_M A dy.$$ When
  $\bar\lambda(t)\le0$ we obtain
  \begin{align*}
    \bar\lambda'(t)&=\lambda'(t)V^{2/n}+\frac{2\lambda}{n}V^{n/2-1}\int_M(-A)dy\\
    &\ge V^{n/2}\left(\lambda'(t)-\frac{8\lambda^2(t)}{n}\right)\\
    &\ge \frac{V^{n/2}}{2}\left(
      \int_M\left(\left|\alpha-2\nabla\nabla\log
          f\right|^2+\Theta(2\nabla\log
        f)\right)f^2dy-\frac{16\lambda^2(t)}{n}\right)\\
  \end{align*}
  where $f$ is the normalized positive eigenfunction corresponding to
  $\lambda$.

  We observe that
  \begin{align*}
    \left|\alpha-2\nabla\nabla\log
      f\right|^2&=\left|\alpha-2\nabla\nabla\log f-\frac 1 n
      \left(A-2\Delta\log f\right)g\right|^2\\ &\qquad+\frac 1 n
    \left(A-2\Delta\log f\right)^2.
  \end{align*}
  Recall that $f$ is the normalized eigenfunction and by H\"older's
  inequality we obtain
  \begin{align}\label{holder inequality}
    \int_M(A-2\Delta \log f)^2 f^2dy&=\int_M(A-2\Delta \log f)^2
    f^2dy\int_M f^2 dy\\ \nonumber &\ge\left(\int_M(A-2\Delta \log
      f)f\cdot f dy\right)^2\\ \nonumber
    &=\left(\int_MAf^2+4\left|\nabla f\right|^2 dy\right)^2\\
    \nonumber &=16\,\lambda^2(t).
  \end{align}
  Finally we have
  \begin{equation*}
    \bar\lambda'(t)\ge 0. \qedhere
  \end{equation*}
\end{proof}

If $\lambda(t)\le0$ we derived indeed the inequality
\begin{align}\label{derivative of the normalized eigenvalue second
    place}
  \bar\lambda'(t)\ge &\frac{V^{2/n}}{2}\left(
    \int_M\left(\left|\alpha-2\nabla\nabla\log f-\frac 1 n
        \left(A-2\Delta\log f\right)g\right|^2+\Theta(2\nabla\log
      f)\right)f^2dy\right)\\\nonumber
  &+\frac{V^{2/n}}{2n}\left(\int_M(A-2\Delta \log f)^2
    f^2dy-\left(\int_M(A-2\Delta \log f)f\cdot f dy\right)^2
    dy\right).
\end{align}

Now we may use \eqref{derivative of the normalized eigenvalue second
  place} to rule out nontrivial expanding breathers.
\begin{theorem}
  Suppose that $(M, g(t))$ is a solution to the abstract geometric
  flow \eqref{metric evolving} with $\alpha$ being scale
  invariant. Assume that $\Theta$ is nonnegative. If $(M, g(t))$ is an
  expanding breather for $t_1<t_2$, then it has to be a gradient
  soliton on $(t_1, t_2)$ in the sense that
  \begin{align*}
    \alpha-2\nabla\nabla\log f-\frac{4\lambda}{n} g=0
  \end{align*}
  where $f$ is the positive normalized eigenfunction corresponding to
  $\lambda(t)$.  Moreover one has $$\Theta(2\nabla\log f)=0.$$
\end{theorem}
\begin{proof}
  Since $\bar\lambda$ is invariant under diffeomorphism and rescaling,
  we have $\bar\lambda(t_1)=\bar\lambda(t_2)$. Since $V(t_1)<V(t_2)$
  there must be a time $t_0\in (t_1,t_2)$ such that $V'(t_0)\ge
  0$. Hence
$$
\lambda(t_0)\le\frac{1}{4V(t_0)}\int_M A(t_0)
dy=-\frac{1}{4V(t_0)}V'(t_0)\le 0.
$$
Proposition \ref{normalized eigenvalue increases whenever negative}
then implies $\bar\lambda(t_1)\le \bar\lambda(t_0)\le0$. Thus, on the
whole interval $[t_1, t_2]$, the function $\bar\lambda(t)$ is
nonpositive increasing and equals at the end points.  This means that
the RHS of \eqref{derivative of the normalized eigenvalue second
  place} vanishes. In particular, the second line of \eqref{derivative
  of the normalized eigenvalue second place} being zero means that
equality holds in H\"older's inequality \eqref{holder inequality}.
Thus $A-2\Delta\log f$ must be a spatial constant which is
$4\lambda(t)$ because $f$ is a normalized eigenfunction corresponding
to $\lambda(t)$.  The vanishing of the first line of \eqref{derivative
  of the normalized eigenvalue second place} means that
\begin{equation*}
  \alpha-2\nabla\nabla\log f-\frac{4\lambda}{n} g=0, \quad \Theta(2\nabla\log
  f)=0.\qedhere
\end{equation*}
\end{proof}


\begin{thebibliography}{99}
\bibitem{AW} Akbar, Mohammad M.; Woolgar, Eric. Ricci solitons and
  Einstein-scalar field theory. Classical Quantum Gravity 26 (2009),
  no. 5, 055015, 14 pp.

\bibitem{Cao1} Cao, Xiaodong. Eigenvalues of $(-\Delta+\frac R 2)$ on
  manifolds with nonnegative curvature operator. Math.~Ann. 337
  (2007), no. 2, 435-441.

\bibitem{Cao2} Cao, Xiaodong. First eigenvalues of geometric operators
  under the Ricci flow. Proc. Amer. Math. Soc. 136 (2008), no. 11,
  4075-4078.

\bibitem{E-K-N-T} Ecker, Klaus; Knopf, Dan; Ni, Lei; Topping,
  Peter. Local monotonicity and mean value formulas for evolving
  Riemannian manifolds.  J.~Reine Angew. Math. 616 (2008), 89-130.

\bibitem{F-I-N} Feldman, Michael; Ilmanen, Tom; Ni, Lei.  Entropy and
  reduced distance for Ricci expanders. J.~Geom. Anal. 15 (2005),
  no. 1, 49-62.

\bibitem{H} Huang, Hong. Optimal transportation and monotonic quantites on evolving manifolds.
Pacific J. Math. 248 (2010) no. 2, 305-316.

\bibitem{K-L} Kleiner, Bruce; Lott, John. Notes on Perelman's
  papers. Geom.~Topol. 12 (2008), no. 5, 2587-2855.

\bibitem{Li} Li, Jun-Fang. Eigenvalues and energy functionals with
  monotonicity formulae under Ricci flow. Math.~Ann. 338 (2007),
  no. 4, 927-946.

\bibitem{L} List, Bernhard. Evolution of an extended Ricci flow
  system. Comm. Anal. Geom. 16 (2008), no. 5, 1007-1048.

\bibitem{M} M\"uller, Reto. Monotone volume formulas for geometric
  flows. J.~Reine Angew. Math. 643 (2010), 39-57.

\bibitem{M2} M\"uller, Reto. Ricci flow coupled with harmonic map
  flow. Ann.\ Sci.\ \'Ec.\ Norm.\ Sup\'er. 45 (2012), no.~1, 101-142.

\bibitem{N} Ni, Lei. The entropy formula for linear heat
  equation. J.~Geom. Anal. 14 (2004), no. 1,
  87-100. Addenda. J.~Geom. Anal. 14 (2004), no. 2, 369-374.

\bibitem{P} Perelman, Grisha. The entropy formula for the Ricci flow
  and its geometric applications. arXiv:math/0211159v1.

\bibitem{reedsimon} Reed, Michael; Simon, Barry. Methods of Modern
  Mathematical Physics, Volume~IV: Analysis of Operators. Academic
  Press, San Diego, 1978.

\bibitem{wu} Wu, Jia Yong. First eigenvalue monotonicity for the
  $p$-Laplace operator under the Ricci flow.
  Acta Math.~Sin.\ (Engl.~Ser.) 27 (2011), no.~8, 1591-1598.

\bibitem{wuwangzheng} Wu, Jia-Yong; Wang, Er-Min; Zheng, Yu. First
  eigenvalue of the $p$-Laplace operator along the Ricci flow.
  Ann.~Global Anal.~Geom.~38 (2010), no.~1, 27-55.

\end{thebibliography}
\end{document}